\documentclass[12pt]{article}
\usepackage{amsmath}
\usepackage{amssymb}
\usepackage{amstext}
\usepackage{amsfonts}
\usepackage{amscd}

\newtheorem{theorem}{Theorem}[section]

\newtheorem{corollary}{Corollary}[section]

\newtheorem{definition}{Definition}[section]
\newtheorem{example}{Example}[section]

\newtheorem{lemma}{Lemma}[section]

\newtheorem{proposition}{Proposition}[section]
\newtheorem{remark}{Remark}[section]

\newenvironment{proof}[1][Proof]{\noindent\textbf{#1.} }{\ \rule{0.5em}{0.5em}}

\input{tcilatex}
\begin{document}
\begin{center}
\textbf{COMPLETENESS PROPERTIES OF TRANSITIVE BINARY RELATIONAL SETS}
\end{center}

\begin{center}
 O. R. Sayed  \\
Department of Mathematics, Faculty of Science, Assiut University,
Assiut 71516, EGYPT \\
 o\_sayed@aun.edu.eg, o\_r\_sayed@yahoo.com
\end{center}

\begin{center}
N. H. Sayed\\
Department of Mathematics, Faculty of  Science,  New Valley University, EGYPT\\
 nabil.gawed@newvedu.au.edu.eg , nhs\_guweid@yahoo.com.
\end{center}

\footnotetext{\textit{Corresponding author}: O. R. Sayed.
}

\begin{abstract}
 The present paper is devoted to study some completeness properties of transitive binary relational set, i.e., a set together with a transitive binary relation (so called t-set).
\end{abstract}

\section{Introduction}

\noindent Abramsky and Jung [1] introduced a method to construct a canonical partially ordered set from a pre-ordered set and said: "Many notions from the theory of order sets make sense even if reflexivity fails". Finally, they sum up these considerations with the slogan: "Order theory is the study of transitive relations". Heckmann [3] introduced and studied the concepts of bounded complete poset, bounded complete domain, finitely complete poset, complete domain, finitarily complete poset, strongly compactly complete domain and compactly complete domain.

This paper is organized as follows: In Section 2, some definitions and results concerning some completeness properties of poset and domain were presented. In Section 3, bounded complete t-sets and bounded complete domain t-sets were introduced and studied. In Section 4, finitely complete t-sets and complete domain t-sets were introduced and studied. In Section 5, we extend the concept of finitary sets in transitive binary relational sets and then introduced and study the concept of finitarily complete t-sets. Finally, in Section 6, strongly compactly complete t-sets and compactly complete t-sets were introduced and studied.

\section{Preliminaries}

\noindent For basic concepts of the poset, we refer to [5]. The concepts of directed subset, domain (directed complete poset), upper cone, upper closure etc., the reader is referred to [1, 3].

\begin{definition}
Let \textit{X} be a nonempty set with a binary relation "$\preccurlyeq $" on \textit{X}.
 The pair $(X, \preccurlyeq)$ is called:

\begin{enumerate}
\item[(1)] a partially ordered set (poset for short) [5] if "$\preccurlyeq $" is reflexive, antisymmetric and transitive;

\item[(2)] a pre-ordered (quasi ordered) set [5] if "$\preccurlyeq $" is reflexive and transitive;

\item[(3)] an equivalence set [5] if "$\preccurlyeq $" is reflexive, symmetric and transitive;

\item[(4)] a continuous information system [4, 6] if "$\preccurlyeq $" is transitive and interpolative (if $\forall x,z\in X$ with $x\preccurlyeq z $ there exists $y\in X$ such that $x\preccurlyeq y\preccurlyeq z)$;

\item[(5)] an abstract base [7] if "$\preccurlyeq $"  is transitive and for every $x\in X$ and every finite subset \textit{A} of \textit{X}, if for every $y\in A, y\preccurlyeq x$,  there exists $z\in X$ such that $y\preccurlyeq z\preccurlyeq x$.
\end{enumerate}

\end{definition}

\begin{definition}
[3]. Let $(X,\preccurlyeq )$ be a domain. A subset $A$ of \textit{X} is called strongly compact if for every $O\in {\tau }_S$ with $A\subseteq O$ there exists a finitary set \textit{F} with $A\subseteq F\subseteq O$, where  ${\tau }_S$ is the Scott-topology on \textit{X}.
\end{definition}

\section{Bounded complete t-sets and bounded complete domain t-sets}

\begin{definition}
A transitive binary relational set (t-set for short) is a pair $(X,\preccurlyeq )$, where \textit{X} is a non-empty set and "$\preccurlyeq$ " is a transitive binary relation on \textit{X}.
\end{definition}

\begin{example}
Partially ordered sets, pre-ordered sets, equivalence sets, continuous information system and abstract bases are t-sets.
\end{example}

\begin{remark}
One can deduce that any abstract base is a continuous information system (Indeed, let $x\in X$ and $\{y\}$ be a finite subset of \textit{X}, where $y\in X$. If $y\preccurlyeq x$, then there exists $z\in X$ such that $y\preccurlyeq z\preccurlyeq x$. Hence  $"\preccurlyeq "$ is interpolative.) but the converse need not be true as we illustrate by the following example:
\end{remark}

\begin{example}
Let $X=\{a,b,x\}$ and $\preccurlyeq =\{(a, a), (b, b), (a, x), (b, x)\}$. Then $"\preccurlyeq "$  is transitive and interpolative. Furthermore, if $A=\{a, b\}$, then $(X,\preccurlyeq )$ is not an abstract base.
\end{example}

\begin{remark}
Every pre-ordered set is an abstract base (Indeed, suppose that $A$ is a finite subset of a pre-ordered set $X$ and such that for every $x\in X$ and for every $ y\in A$, $ y\preccurlyeq x$. Then $y\preccurlyeq x\preccurlyeq x$.). The converse need not be true as we illustrate by the following example:
\end{remark}

\begin{example}
Let $X=\{a,b,c,d,e\}$ and $\preccurlyeq  =\{(a,a)\}$. Then $(X,\preccurlyeq )$ is an abstract base. It is clear that $"\preccurlyeq "$ is not reflexive. Hence $(X,\preccurlyeq )$ is not a pre-ordered set.
\end{example}

\begin{definition}
  Let $(X,\preccurlyeq )$ be a t-set and $A\subseteq X$.
\begin{enumerate}

\item[(1)]  The lower (resp. upper) bounded subset in \textit{X} of \textit{A} is denoted
 by $lb(A)$ (resp.$\ ub(A)$) and defined as follows:\\
 $lb(A)=\{x\in X:\forall y\in A, x\preccurlyeq y\}$ (resp.$ub(A)=\{x\in X:\forall y\in A,y\preccurlyeq x\})$.

 Each element in $lb(A)$ (resp. $ub\ (A)$) is called a lower (resp. an upper) bound of \textit{A}.
\item[(2)]  The subset of least (resp. largest) elements of a subset \textit{A} is denoted by $le(A)$ (resp. $la(A)$) and defined as follows:

$le(A)=\{x\in A: \forall  y\in A, x\preccurlyeq y\}$ (resp. $la(A)=\{x\in A: \forall  y\in A,
 y\preccurlyeq x\}).$
Each element in $le(A)$ (resp. $la(A))$ is called a least (resp. a largest) element of \textit{A}.

\item[(3)] The infimum (resp. supremum) subset in \textit{X} of \textit{A} is denoted by $\inf  (A)$ (resp. $\sup  (A)$) and defined as follows:

 $\inf  (A)=la(lb(A))$ (resp. $\sup  (A)=le(ub(A))$).

 Each element in $\inf  (A)$ (resp. $\sup  (A)$) is called an infimum (resp. a supremum) of\textit{ A}.

\item[(4)] The lower (resp. upper) closure in \textit{X} of \textit{A} is denoted by $\downarrow (A)$ (resp. $\uparrow (A)$) and defined as follows:

 $\downarrow (A)=\{x\in X:$ there exists$ y\in A,x\preccurlyeq y\}$ (resp.$\uparrow  (A)=\{x\in X:$ there exists $ y\in A, y\preccurlyeq x\})$.
\end{enumerate}
\end{definition}

\begin{definition}
 Let $(X,\preccurlyeq )$ be a t-set and $A\subseteq X$. Then \textit{A} is called:

\begin{enumerate}
\item[(1)] a directed subset if  $A \neq \phi$ and for every distinct elements  \textit{x, y} in \textit{A}, there exists $ z\in A\cap ub(\{x\ , y\})$;
\item[(2)] an upper cone if there exists $x\in A$ such that $A=\uparrow x$.
\end{enumerate}
\end{definition}

\begin{definition}
A t-set $(X,\preccurlyeq )$ is called bounded complete if \textit{X} is an upper cone and for every
$x, y\in X, ub(\{x,y\})$ is empty or an upper cone.
\end{definition}

\begin{theorem}
For a t-set $(X,\preccurlyeq )$, the following statements are equivalent:
\begin{enumerate}
\item[(1)]  \textit{X} is a bounded complete t-set;

\item[(2)] $le(X)\neq \phi$ and for every $x, y\in X$ with $ub(\{x, y\})\neq \phi$,
${\sup  (\{x, y\})\neq \phi }$ ;

\item[(3)] If \textit{A} is a finite bounded subset from above, then $\sup  (A) \neq \phi$;

\item[(4)]  If \textit{A} is a finite subset of \textit{X}, then $ub(A)$ is either $\phi$ or an upper cone.
\end{enumerate}
\end{theorem}

\begin{proof}
 $(1) \Rightarrow (2)$ : Since \textit{X} is an upper cone, then there exists $a\in X$ such that  $\uparrow a=X$.  So, $\{a\}\subseteq le(X)$. If $ub(\{x, y\})\neq \phi$, then $ub(\{x, y\})$ is an upper cone. Hence there exists $b\in ub(\{x, y\})$ such that $\uparrow b = ub(\{x,y\})$. So, $b\in le(ub(\{x, y\}))$. Therefore
  $\sup  (\{x, y\})\neq \phi$.

 $(2) \Rightarrow (3)$ : Now $\phi$ is a finite bounded set from above since $ub(\phi)=X\neq \phi$. Since
 $le(X) = le(ub(\phi))\neq \phi$, then $\sup  (\phi)\neq \phi$. Let \textit{A} be a non-empty finite bounded subset from above. If $A=\{z\}$ and $ub(\{z\}) \neq \phi$, then $\sup  (A)\neq \phi$. Suppose
 $A=\{x_1, x_2, x_3,..., x_n\}$ and $ub(A)\neq \phi$. Now $A_{ 1, 2} = \{x_1, x_2\}$ and
 $ub(A_{1, 2})\neq \phi$, then $\sup  (A_{1, 2})\neq \phi$. Take $u_{1,2}\in \sup  (A_{1, 2})$ and
 consider $A_{ 1,2,3} = \{u_{1, 2}, x_3\}$. Then $\sup  (A_{1, 2, 3}) \neq \phi$ because $ub(A_{1,
 2, 3})\neq \phi$. We can proceed until consider the set $B=\{u_{1, 2, ... ,n-1}, x_n\}$. Since
 $ub(B)\neq \phi$, then $\sup  (B) \neq \phi$. Now, for every $l\in \sup  (B)$, $l\in ub(A)$. Since $m\in
 ub(A)$, one can deduce that $l\preccurlyeq m$. Then $l\in \sup  (A)$. So, $\sup  (A) \neq \phi$.

 $(3) \Rightarrow (4)$ : Let \textit{A} be a finite subset of \textit{X}. If \textit{A} is not bounded from above, then $ub(A)=\phi$. Let \textit{A} be finite bounded subset from above. Then $\sup  (A) = le(ub(A))\neq \phi$. Then there exists $x\in ub(A)$ such that  $\uparrow x = ub(A)$.

 $(4) \Rightarrow (1)$ : Now, $X = ub(\phi)$ and so \textit{X} is an upper cone. Since for every $x, y\in X$,
 $\{x,  y\}$ is finite. Then $ub(\{x,y\})=\phi$ or $ub(\{x,y\})$ is an upper cone.
\end{proof}

\begin{lemma}
For a t-set $(X,\preccurlyeq )$, the following statements are equivalent:

\begin{enumerate}
\item[(1)] If \textit{A} is bounded subset from above, then $\sup  (A) \neq \phi$ ;
\item[(2)] If \textit{A} is a non-empty subset of \textit{X}, then $\inf  (A) \neq \phi$.
\end{enumerate}
\end{lemma}

\begin{proof}
   $(1)\Rightarrow(2)$: Let \textit{A} be non-empty subset of \textit{X} and $B=lb(A)$. Now, $ub(B)\supseteq A\neq \phi$. Then $\sup  (B) \neq \phi$. Let $x\in \sup  (B)$. Now, $x\in ub(B)$. Then for every $a\in A$, $x\preccurlyeq a$. Then $x\in la(lb(A))=\inf  (A)$. Hence $\inf  (A) \neq \phi$.

$(2)\Rightarrow(1)$: Suppose that \textit{A} is a bounded subset of \textit{X} from above and $B=ub(A)\neq \phi$. Then $\inf  (B) \neq \phi$. Let $x\in \inf  (B)$. Since $A\subseteq lb(B)$ and $x\in \inf  (B)$, then for every $a\in A$, $a\preccurlyeq x$. Thus $x\in le(ub(A))=\sup  (A)$. Therefore $\sup  (A) \neq \phi$.
\end{proof}

\begin{definition}
 A t-set $(X,\preccurlyeq )$ is called a bounded complete domain if it is bounded complete and domain.
\end{definition}

\begin{theorem}
 For a domain t-set $(X,\preccurlyeq )$, the following statements are equivalent:

\begin{enumerate}
\item[(1)] \textit{X} is a bounded complete t-set;

\item[(2)]  $le(X)\neq \phi$ and $\forall  x, y\in X$ with $ub(\{x, y\})\neq \phi$,
$\sup  (\{x, y\})\neq \phi$;

\item[(3)]  If \textit{A} is a finite bounded subset from above, $\sup  (A) \neq \phi$;

\item[(4)]  If \textit{A} is a finite subset of \textit{X}, $ub(A)$ is either $\phi$ or an upper cone;

\item[(5)]  If \textit{A} is bounded subset from above, $\sup  (A) \neq \phi$;

\item[(6)]  If \textit{A} is a non-empty subset of \textit{X}, $\inf  (A) \neq \phi$.
\end{enumerate}
\end{theorem}

\begin{proof}
From Theorem 3.1 and Lemma 3.1, it rests to prove that (3) and (5) are equivalent.

$(3)\Rightarrow(5)$: Let \textit{A} be a bounded subset of \textit{X} from above and $D=\{x:x$ is a fixed element of $\sup  (F)$ for every finite subset \textit{F} of \textit{A}\}. Since $\sup  (\phi) \neq \phi$ and for every $y\in \sup  (F_1\cup F_2),  y\in ub(\sup  (F_1) \cup \sup  (F_2))$, where $F_1$ and $F_2$ are finite subsets of \textit{A} , then \textit{D} is directed. Thus $\sup  (D) \neq \phi$. Now, for every $l\in \sup  (D)$, $l\in ub(A)$. Suppose that $z\in ub(A)$. Then for all $m\in A, m\preccurlyeq z$ so that $z\in ub(A)$. Thus $l\preccurlyeq z$ so that
 $ l\in \sup  (A)$. Hence $\sup  (D)\subseteq \sup  (A)$. Therefore $\sup  (A)\neq \phi$.

$(5)\Rightarrow(3)$: Obvious.
\end{proof}

\section{Finitely complete t-sets and complete domain t-sets}

\begin{definition}
 A t-set $(X,\preccurlyeq )$ is called finitely complete if \textit{X} is an upper cone, and for all $x, y\in X, ub(\{x,y\})$ is an upper cone.
\end{definition}

One can easily deduce that any finitely complete t-set is a bounded complete t-set.

\begin{theorem}
For a t-set $(X,\preccurlyeq )$, the following statements are equivalent:
\begin{enumerate}
\item[(1)]  \textit{X} is a finitely complete t-set;

\item[(2)]  \textit{X} has a least element and for all $x, y\in X$, $\sup (\{x, y\}) \neq \phi$;

\item[(3)]  If \textit{A} is a finite subset of \textit{X}, then $\sup  (A) \neq \phi$;

\item[(4)]If A is a finite subset of \textit{X}, then $ub(A)$ is an upper cone.
\end{enumerate}
\end{theorem}

\begin{proof}
$(1)\Rightarrow (2)$: Since \textit{X} is an upper cone, then there exists $a\in X$ such that
$\uparrow a = X$. So, $a\in le(X)$. Suppose that $x, y\in X$. Then there exists $z\in ub(\{x,y\})$ such that $\uparrow z = ub(\{x,y\})$. Therefore $z\in \sup  (\{x, y\})$.

$(2)\Rightarrow (3)$: First, the empty set is finite. Since there exists $x\in le(X)$, then there exists $x\in \sup  (\phi)$. Suppose that $A=\{z\}$. Now, we have that  $\sup  (\{z\}) = \sup  (\{z,z\}) \neq \phi$. Let $A=\{x_1, x_2, x_3, ... , x_n\}$, i.e. \textit{A} is a finite set. Now, $A_{ 1,2}=\{x_1,x_2\}$, then there exists $\ u_{1,2}\in \sup  (A_{1,2})$. Put $A_{1,2,3}=\{u_{1,2},x_3\}$ so that there exists
$ u_{1,2,3}\in \sup  (A_{1,2,3})$. We can proceed until consider the set $B=\{u_{1,2, ... ,n-1},x_n\}$ so that there exists $l\in \sup  (B)$. Then $l\in ub(A)$. Let $m\in ub(A)$. One can deduce that $l\preccurlyeq m$. Therefore $ l\in \sup  (A)$.

$(3)\Rightarrow (4)$ : Let \textit{A} be a finite set. Then $\sup  (A) \neq \phi$. Thus, there exists
$l\in le(ub(A))$ so that $\uparrow l=ub(A)$. Therefore $ub(A)$ is an upper cone.

$(4)\Rightarrow (1)$ : Since $\phi$ is finite and $ub(\phi)=X$, then \textit{X} is an upper cone. Since the set $\{x,y\}$ is finite for every $x,y \in X$, then $ub(\{x,y\})$ is an upper cone.
\end{proof}

\begin{definition}
$(X,\preccurlyeq )$ is called a complete domain t-set if it is finitely complete t-set and domain t-set.
\end{definition}

\begin{theorem}
 For a t-set $(X,\preccurlyeq )$, the following statements are equivalent:
\begin{enumerate}
\item[(1)]  \textit{X }is a complete domain t-set;

\item[(2)]  \textit{X} is a bounded complete domain t-set with $la(X)\neq \phi$;

\item[(3)]  If \textit{A} is a subset of \textit{X}, $\inf  (A) \neq \phi$;

\item[(4)]  If \textit{A} is a subset of \textit{X}, $\sup  (A) \neq \phi$;

\item[(5)]  If \textit{A} is a finite subset of \textit{X} or a directed subset of \textit{X},
$\sup  (A) \neq \phi$.
\end{enumerate}
\end{theorem}

\begin{proof}
$(1)\Rightarrow(2)$: Any complete domain t-set is bounded complete domain. Now, since for all
$x, y\in X, ub(\{x,y\})$ is an upper cone, then $ub(\{x,y\})\neq \phi$. Hence \textit{X} is directed. Therefore $la(X)=\sup  (X) \neq \phi$.

$(2)\Rightarrow(3)$: Let \textit{A} be a subset of \textit{X}. First, if $A=\phi$, then $X=lb(\phi)$. Since $le(X)\neq \phi$, then there exists $l\in \inf  (\phi)$. Second, if $A\neq \phi$, then from Theorem 3.2(6), $\inf  (A) \neq \phi$.

$(3)\Rightarrow(4)$: Let \textit{A }be a subset of \textit{X}. Since $\inf  (\phi) \neq \phi$, then there exists $l\in la(X)$ so that every subset of \textit{X} is bounded from above. From Lemma 3.1,
$ \sup(A)\neq \phi$ ;

$(4)\Rightarrow(5)$: Obvious.

$(5)\Rightarrow(1)$: Since for every directed subset \textit{A} of \textit{X}, then
$\sup  (A) \neq \phi$.  Hence \textit{X} is a domain t-set. From Theorem 4.1 (3), \textit{X} is a finitely complete t-set.
\end{proof}

\section{Finitarily complete t-sets}

\begin{definition}
Let $(X,\preccurlyeq )$ be a t-set. A subset \textit{A} of \textit{X} is called finitary if there exists a finite subset \textit{F} of \textit{A} with $A\subseteq \uparrow (F)$.
\end{definition}

\begin{proposition}
Let $(X,\preccurlyeq )$ be a t-set and $\{A_j:j\in \{1,2,...,n\}\}$ be a family of finitary subsets of \textit{X}. Then $\bigcup^n_{j=1}{A_j}$ is a finitary subset.
\end{proposition}

\begin{proof}
Since for every $ j\in \{1,2,...,n\}$ there exists a finite subset $K_j$ such that $K_j\subseteq A_j \subseteq \uparrow (K_j),$ then $\bigcup^n_{j=1}{K_j}\subseteq \bigcup^n_{j=1}{A_j}\subseteq \bigcup^n_{j=1}{\uparrow (K_j)}\subseteq \uparrow (\bigcup^n_{j=1}{K_j})$. Since $\bigcup^n_{j=1}{K_j}$ is finite, then it is clear that  $\bigcup^n_{j=1}{A_j}$ is finitary.
\end{proof}

\begin{definition}
 A t-set $(X,\preccurlyeq )$ is called finitarily complete if   \textit{X} is finitary,
  $\forall  x, y\in X, ub(\{x,y\})$ is finitary.
\end{definition}

\begin{theorem}
Let $(X,\preccurlyeq )$ be a t-set. Then the following statements are equivalent:
\begin{enumerate}
\item[(1)]  \textit{X }is finitarily complete;

\item[(2)]  \textit{X} is finitary and if \textit{A} and \textit{B} are finitary upper sets, then $A\cap B$ is finitary;

\item[(3)]  If $A_1,..., A_n$ are finitary subsets of \textit{X}, then $\bigcap^n_{j=1}{A_j}$ is finitary;

\item[(4)]  If \textit{B} is finite subset of \textit{X}, then $ub(B)$ is finitary.
\end{enumerate}
\end{theorem}

\begin{proof}
 $(1)\Rightarrow(2)$:  If \textit{X} is finitarily complete, then \textit{X} is finitary. If \textit{A} is finitary upper set, then there exists a finite set $F_1\subseteq A$ such that $A\subseteq \uparrow (F)$ and $\uparrow (A)\subseteq A$. Hence $A=\uparrow (F_1)$ and if \textit{B} is finitary upper set, then there exists a finite set $F_2\subseteq B$ such that $B\subseteq \uparrow (F_2)$ and $\uparrow (B)\subseteq B$. Hence $B=\uparrow (F_2)$. Thus $A\cap B=\uparrow (F_1)\cap \uparrow (F_2)=(\bigcup_{a\in F_1}{(\uparrow a)})\bigcap (\bigcup_{b\in F_2}{(\uparrow b)})=\bigcup_{a\in F_1,b\in F_2}{(\uparrow a \cap \uparrow b)}$. So, $A\cap B$ is a finite union of finitary sets. Therefore $A\cap B$ is finitary.

 $(2)\Rightarrow(3)$:  By indication. The empty intersection is \textit{X}.

 $(3)\Rightarrow(4)$:  If \textit{B} is finite, then $ub(B)=\bigcup_{e\in B}{(\uparrow e)}$ upper cones are
 finitary. This is finite intersection of finitary sets.

 $(4)\Rightarrow(1)$:  \textit{X} is the set of upper bounds of $\phi$ , and $\uparrow x \cap \uparrow y$ is the set of upper bounds of $\{x,y\}$.
\end{proof}

\begin{theorem}
\begin{enumerate}
\item[(1)] Every finite pre-ordered set $(X,\preccurlyeq )$ is finitarily complete;

\item[(2)]  Every bounded complete $(X,\preccurlyeq )$ pre-ordered set is finitarily complete.
\end{enumerate}
\end{theorem}

\begin{proof}

\begin{enumerate}
\item[(1)] Since \textit{X} is finite and $ub(\{x,y\})$ for each $x, y\in X$ is finite also, then one can easily deduce that \textit{X} is finitarily complete.

\item[(2)]  Since \textit{X} is upper cone, then $X=\uparrow \{x\}$ for some $x\in X$. Hence \textit{X} is finitary. Also, one can deduce that $ub(\{x,y\})$ for each $x, y\in X$ is a finitary subset of \textit{X} because $ub(\{x,y\})$ is empty or upper cone. Therefore, \textit{X }is finitarily complete.
\end{enumerate}
\end{proof}

The following two examples illustrate that the concepts of bounded completeness and finiteness are independent notions for t-sets (moreover for posets).

\begin{example}
Let $X=\{a,b,c,d\}$ and $\preccurlyeq =\{(a,a),(b,b),(c,c),(d,d)\}$. Then $(X,\preccurlyeq )$ is finite poset but not bounded complete.
\end{example}

\begin{example}
Let $N=\{1,2,3,...\}$ and $\preccurlyeq $ be the usual partially ordered relation on \textit{N}. Then $(N,\preccurlyeq )$ is bounded complete but not finite.
\end{example}

\section{Strongly compactly complete t-sets}

\begin{definition}
A triple $(X,\preccurlyeq ,\tau )$ is called a topological t-set, where $(X,\preccurlyeq )$ is a t-set and $(X,\tau )$ is a topological space.
\end{definition}

\begin{definition}
Let $(X,\preccurlyeq ,\tau )$ be a topological t-set. A subset \textit{A} of \textit{X} is called strongly compact if for all $O\in \tau $ such that $A\subseteq O$, there exists a finitary subset \textit{F} of \textit{X} such that $A\subseteq F\subseteq O$.
\end{definition}

\begin{theorem}
Let $(X,\preccurlyeq ,\tau )$ be a topological t-set such that each member of $\tau $ is an upper subset. If a subset \textit{A} of \textit{X} is strongly compact, then \textit{A} is compact.
\end{theorem}

\begin{proof}
Let $\Re$ be an open cover of \textit{A}, i.e. $A\subseteq \bigcup_{B\in \Re}{B}$ and
 $\Re \subseteq \tau $. Put $\bigcup_{B\in \Re}{B}=G$. Then $A\subseteq G\in \tau $. Since \textit{A} is strongly compact, then there exists a finitary subset \textit{K} of \textit{G} such that $A\subseteq K\subseteq G$ so that there exists a finite subset \textit{F} of \textit{K} such that $K\subseteq \uparrow (F)$. Then for every $x\in F$ there exists $B_x\in \Re$ such that $x\in B_x$. So, $F\subseteq \bigcup_{x\in F}{B_x}$. Hence $A\subseteq K\subseteq \uparrow (F)\subseteq \uparrow (\bigcup_{x\in F}{B_x})=\bigcup_{x\in F}{B_x}$. Therefore, \textit{A} is compact.
\end{proof}

\begin{corollary}
\begin{enumerate}
\item[(1)] If \textit{A} is strongly compact subset of \textit{X}  with respect to  the topological t-set  $(X,\preccurlyeq ,{\tau }_{Alx})$, then \textit{A} is compact, where ${\tau }_{Alx}$ is the Alexandroff topology induced by $"\preccurlyeq "$;

\item[(2)]  If \textit{A} is a strongly compact subset of \textit{X} with respect to the topological t-set $(X,\preccurlyeq ,{\tau }_{s^*})$ then \textit{A} is compact, where ${\tau }_{S^*}$  is the Scott*-topology induced by $"\preccurlyeq "$.
\end{enumerate}
\end{corollary}

\begin{theorem}
Let $(X,\preccurlyeq )$ be a t-set and $\{A_j:j\in \{1,2,...,n\}\}$ be a family of finitary subsets of \textit{X}. Then $\bigcup^n_{j=1}{A_j}$ is a finitary subset.
\end{theorem}

\begin{proof}
Since for every $j\in \{1,2,...,n\}$ there exists a finite subset $K_j$ such that $K_j\subseteq A_j\subseteq \uparrow (K_j)$, then $\bigcup^n_{j=1}{K_j}\subseteq \bigcup^n_{j=1}{A_j}\subseteq \bigcup^n_{j=1}{\uparrow (K_j)}\subseteq \uparrow (\bigcup^n_{j=1}{K_j})$. Since $\bigcup^n_{j=1}{K_j}$ is finite, then it is clear that  $\bigcup^n_{j=1}{A_j}$ is finitary.
\end{proof}

\begin{theorem}
Let $(X,\preccurlyeq ,\tau )$ be a topological t-set and $\{A_j:j\in \{1,2,...n\}\}$ be a family of strongly compact subsets. Then $\bigcup^n_{j=1}{A_j}$ is strongly compact subset.
\end{theorem}

\begin{proof}
Suppose $O\in \tau $  such that $\bigcup^n_{j=1}{A_j}\subseteq O$. Then for all $j\in J$ there exists a finitary subset $B_j$ such that $A_j\subseteq B_j\subseteq O$ so that $\bigcup^n_{j=1}{A_j}\subseteq \bigcup^n_{j=1}{B_j}\subseteq O$. From Theorem 6.2, $\bigcup^n_{j=1}{B_j}$ is finitary. Hence $\bigcup^n_{j=1}{A_j}$ is strongly compact.
\end{proof}

\begin{definition}
Let$(X,\preccurlyeq ,\tau )$ be a topological t-set. \textit{X} is called strongly compactly complete t-set if \textit{X} is strongly compact and for every $x,y\in X,ub(\{x,y\})$ is strongly compact.
\end{definition}

\begin{theorem}
Let $(X,\preccurlyeq ,\tau )$ be a topological t-set. Consider the following statements:

\begin{enumerate}
\item[(1)]  \textit{X} is strongly compactly complete;

\item[(2)]  \textit{X} is finitary and the intersection of two finitary upper sets is strongly compact. Then:
\end{enumerate}
\noindent (A)  (1) $\Rightarrow $ (2).

\noindent (B) If  $"\preccurlyeq "$  is reflexive, then (2) $\Rightarrow $ (1).

\end{theorem}

\begin{proof}

\noindent(A) Since \textit{X} is strongly compact and open, then there exists a finitary set \textit{B} of \textit{X} such that $X\subseteq B\subseteq X$. So, \textit{X} is finitary. Suppose \textit{A} and \textit{B} be two finitary upper sets. Then there are finite sets \textit{E} and \textit{F} such that $A\subseteq \uparrow (E)\subseteq \uparrow (A)\subseteq A$ and $B\subseteq \uparrow (F)\subseteq \uparrow (B)\subseteq B$. Hence, we have $A=\uparrow (E)$ and $B=\uparrow (F)$. Now,  $A\cap B=\uparrow (E) \cap \uparrow (F)=\bigcup_{e\in E,f\in F}{(\uparrow e\cap \uparrow f)}$. Therefore, from Theorem 6.2, $A\cap B$ is strongly compact.

\noindent (B)  Since \textit{X} is finitary and the only open set containing \textit{X}  is \textit{X} itself, then \textit{X}  is strongly compact. Let $x,y\in X$. Since $"\preccurlyeq "$ is reflexive, then for every $\ x\in X,\uparrow \{x\}$ is finitary and since $"\preccurlyeq "$ is transitive, then $\uparrow x$ is upper set. Hence $\uparrow \{x\} \cap \uparrow \{y\}$ is strongly compact.
\end{proof}

\begin{definition}
Let $(X,\preccurlyeq ,\tau )$ be a topological t-set. \textit{X} is called compactly complete if\textit{ X} is compact and for all $x,y\in X,ub(\{x,y\})$ is compact.
\end{definition}

\begin{theorem}
Let $(X,\preccurlyeq ,\tau )$ be a topological t-set. Consider the following statements:

\begin{enumerate}
\item[(1)]  \textit{X} is compactly complete;

\item[(3)]   \textit{X} is finitary and the intersection of two finitary upper sets is compact. Then:
\end{enumerate}

\noindent (A)  If $\tau $ has the property \textit{F}, then $(1)\Rightarrow (2)$ ;

\noindent (B)  If $"\preccurlyeq "$ is reflexive, and each member of $\tau $ is an upper set then $(2)\Rightarrow (1)$.

\end{theorem}

\begin{proof}
\noindent  (A) since $\tau $ has the property \textit{F}, then \textit{X} is finitary. Let \textit{U} and \textit{V} be two finitary upper sets. Then there are finite sets \textit{E} and \textit{M} such that $U\subseteq \uparrow (E)\subseteq \uparrow (U)\subseteq U$ and  $V\subseteq \uparrow (M)\subseteq \uparrow (V)\subseteq V$ .  So, $U=\uparrow (E)$ and  $V=\uparrow (M)$. Now, $U\cap V=\uparrow (E) \cap \uparrow (M)=\bigcup_{e\in E,m\in M}{(\uparrow e \cap \uparrow m)}$. Therefore  $U\cap V$ is compact because a finite union of compact subsets is compact.

\noindent (B) Since \textit{X} is finitary, then \textit{X} is strongly compact. From Theorem 6.1, \textit{X} is compact. Since $"\preccurlyeq "$  is reflexive, then for all $x\in X, x\in \uparrow x$. Thus, we have $\uparrow x$ is finitary. Furthermore, for all $x\in X, \uparrow x$ is an upper set. Therefore, $\uparrow x \cap \uparrow y$ is compact for all $x, y\in X$.
\end{proof}

\textbf{Conclusion:} Theorems 3.1, 3.2, 4.1, 4.2, 5.1, 6.1, 6.2, 6.3, 6.4, Proposition 5.1, Lemma 5.1 and Corollary 6.1 can be obtained if we replace the condition of t-set by a pre-ordered set (resp. abstract base, continuous information system)

\end{document}